\documentclass[10pt]{amsart}
\usepackage[spanish,english]{babel}
\usepackage{amsthm,amsmath,amssymb}
\usepackage{xypic}
\usepackage{mathrsfs}

\newtheorem*{A}{Main Theorem}

\newtheoremstyle{theorem}
     {11pt}
     {11pt}
     {}
     {}
     {\bfseries}
     {}
     {.5em}
     {\noindent\thmnumber{#2}. \thmname{#1}\thmnote{#3}}
\theoremstyle{theorem}

\newtheorem{thm}{Theorem}[section]
\newtheorem{lemma}[thm]{Lemma}
\newtheorem{propo}[thm]{Proposition}
\newtheorem{remark}[thm]{Remark}
\newtheorem{fact}[thm]{Fact}
\newtheorem{coro}[thm]{Corollary}
\newtheorem{ques}[thm]{Question}

\newtheoremstyle{claim}
     {11pt}
     {11pt}
     {}
     {}
     {\bfseries}
     {}
     {.5em}
     {\noindent\thmname{#1} \thmnote{#3}.}
\theoremstyle{claim}

\newtheorem*{claim}{Claim}

\newtheoremstyle{example}
     {11pt}
     {11pt}
     {}
     {}
     {\bfseries}
     {.}
     {.5em}
     {\noindent\thmnumber{#2}. \thmname{#1}\thmnote{#3}}

\theoremstyle{example}
\newtheorem{ex}[thm]{Example}
\newtheorem{theproof}[thm]{Proof of the Main Theorem}

\newcommand{\CL}[1]{\mathnormal{CL}(#1)}
\newcommand{\K}[1]{\mathcal{K}(#1)}
\newcommand{\F}[2][\empty]{\mathcal{F}_{#1}(#2)}
\newcommand{\f}{\mathbf{F}}
\newcommand{\la}{\langle}
\newcommand{\ra}{\rangle}
\newcommand{\h}{\mathcal{H}}
\newcommand{\N}{\mathbb{N}}
\newcommand{\co}[1]{\mathnormal{CO}(#1)}
\newcommand{\Q}{\mathbb{Q}}
\newcommand{\q}[3][\empty]{\mathcal{Q}^{#1}\left(#2,#3\right)}
\newcommand{\qs}[1]{\mathcal{Q}(#1)}
\newcommand{\nc}[2]{\mathfrak{nc}(#1,#2)}
\newcommand{\inte}[2][X]{\mathrm{int}_{#1}\!\left(#2\right)}
\newcommand{\cl}[2][X]{\mathrm{cl}_{#1}\!\left(#2\right)}
\newcommand{\C}{\mathcal{C}}
\newcommand{\D}{\mathcal{D}}
\newcommand{\s}{\mathcal{S}}
\newcommand{\U}{\mathcal{U}}
\newcommand{\V}{\mathcal{V}}
\newcommand{\cs}{{}\sp\omega2}
\newcommand{\filt}{\mathfrak{F}}
\newcommand{\cont}{\mathfrak{c}}

\title{Disconnectedness properties of Hyperspaces}

\author[R. Hern\'andez-Guti\'errez]{Rodrigo Hern\'andez-Guti\'errez}
	\address[R. Hern\'andez-Guti\'errez]{Instituto de Matem\'aticas, Universidad Nacional
Aut\'onoma de M\'exico, Ciudad Universitaria, M\'{e}xico D.F., 04510, M\'exico}
	\email[R. Hern\'andez-Guti\'errez]{rod@matem.unam.mx}
	\thanks{This paper is part of the first author's doctoral dissertation directed by the second author. Research was supported by PAPIIT grant IN-102910 and CONACyT scholarship for Doctoral Students.}

\author[A. Tamariz-Mascar\'ua]{Angel Tamariz-Mascar\'ua}
	\address[A. Tamariz-Mascar\'ua]{Departamento de Matem\'aticas, Facultad de Ciencias, Universidad Nacional Aut\'onoma de M\'exico,
Ciudad Universitaria, M\'{e}xico D.F., 04510, M\'exico}
	\email[A. Tamariz-Mascar\'ua]{atamariz@servidor.unam.mx}

\begin{document}

\date{September 29, 2011}
\subjclass[2010]{54B20, 54G05, 54G10, 54G12, 54G20}
\keywords{hyperspaces, Vietoris topology, $F\sp\prime$-space, $P$-space, hereditarily disconnected}

\begin{abstract}
Let $X$ be a Hausdorff space and let $\h$ be one of the hyperspaces $\CL{X}$, $\K{X}$, $\F{X}$ or $\F[n]{X}$ ($n$ a positive integer) with the Vietoris topology. We study the following disconnectedness properties for $\h$: extremal disconnectedness, being a $F\sp\prime$-space, $P$-space or weak $P$-space and hereditary disconnectedness. Our main result states: if $X$ is Hausdorff and $F\subset X$ is a closed subset such that $(a)$ both $F$ and $X-F$ are totally disconnected, $(b)$ the quotient $X/F$ is hereditarily disconnected, then $\K{X}$ is hereditarily disconnected. We also show an example proving that this result cannot be reversed.
\end{abstract}

\maketitle

Given a $T_1$ space $X$, let $\CL{X}$ be the hyperspace of nonempty closed subsets of $X$ with the Vietoris topology. Let us consider the following hyperspaces
\begin{eqnarray}
\K{X}&=&\{A\in \CL{X}:A\textrm{ is compact}\},\nonumber\\
\F{X}&=&\{A\in \CL{X}: A \textrm{ is finite}\},\nonumber\\
\F[n]{X}&=&\{A\in \CL{X}:|A|\leq n\}\textrm{ for $n$ a positive integer.}\nonumber
\end{eqnarray}

The study of the Vietoris topology on hyperspaces was first motivated by Ernest Michael's outstanding paper \cite{michael}. Concerning connectedness properties, Michael stated the following:

\begin{thm}\cite[Theorem 4.10]{michael}\label{michael1}
Let $X$ be a $T_1$ space and $\F{X}\subset\h\subset \CL{X}$. If any of $X$, $\F[n]{X}$ ($n$ a positive integer) or $\h$ are connected, then the following spaces are also connected: $X$, $\F[m]{X}$ for each positive integer $m$ and every $\h'$ satisfying $\F{X}\subset\h'\subset \CL{X}$.
\end{thm}

Recall that a space $X$ is \emph{zero-dimensional} if for every closed subset $A\subset X$ and $x\in X-A$, there is a clopen set $O$ such that $x\in O$ and $O\cap A=\emptyset$. We say that $X$ is \emph{totally disconnected} if for every pair of points $x,y\in X$ with $x\neq y$, there is a clopen set $O$ such that $x\in O$, $y\notin O$.

\begin{thm}\cite[Proposition 4.13]{michael}\label{michael2}
For a $T_1$ space $X$ we have: 
\begin{itemize}
\item $X$ is zero-dimensional if and only if $\K{X}$ is zero-dimensional,
\item $X$ is totally disconnected if and only if $\K{X}$ is totally disconnected,
\item $X$ is discrete if and only if $\K{X}$ is discrete,
\item $X$ has no isolated points if and only if $\CL{X}$ has no isolated points.
\end{itemize}
\end{thm}

In this paper, we present similar results about other classes of disconnected spaces. Most of our results will be in the realm of Hausdorff spaces. Tychonoff spaces will be used when the classes of spaces considered require so.

First we will consider classes of highly disconnected spaces. If $X$ is a space and $p\in X$, we call $p$ a \emph{$P$-point} of $X$ if $p$ belongs to the interior of every $G_\delta$ set that contains it. We say that $X$ is a \emph{$P$-space} if all its points are $P$-points of $X$. For properties of $P$-spaces, see problem 1W of \cite{porterwoods}. Every regular $P$-space is zero-dimensional, so being a $P$-space is a stronger condition than zero-dimensionality in the realm of regular spaces. In Sections \ref{ppoints} and \ref{fspaces}, we study when a hyperspace can be a $P$-space using other classes of spaces such as $F$-spaces. After that, in section \ref{compactisfinite} we give remarks on spaces in which compact subsets are finite ($P$-spaces are of this kind by Remark \ref{fspacefinite}).

On the other hand, we consider a property roughly weaker than total disconnectedness. We call a space $X$ hereditarily disconnected if every nonempty connected subset of $X$ is a singleton. Clearly, every totally disconnected space is hereditarily disconnected but there are examples (given below) that show these classes do not coincide. In \cite[83.5]{ill-nad}, Illanes and Nadler ask whether $\CL{X}$ or $\K{X}$ are hereditarily disconnected when $X$ is hereditarily disconnected and metrizable. In \cite{pol}, E. Pol and R. Pol answer this in the negative and make some interesting remarks. In Section \ref{herdisc}, we extend the work of E. Pol and R. Pol and give some examples. Our main result is

\begin{A}
Let $X$ be a Hausdorff space. Assume that there is a closed subset $F\subset X$ such that
\begin{itemize}
\item[(a)] both $F$ and $X-F$ are totally disconnected,
\item[(b)] the quotient $X/F$ is hereditarily disconnected.
\end{itemize}
Then $\K{X}$ is hereditarily disconnected.
\end{A}

We also show that the statement of the Main Theorem cannot be reversed by giving an example in Section \ref{example}.

\section{Preliminaries}

We denote by $\N$ the set of positive integers, $\omega=\N\cup\{0\}$, the unit interval $I=[0,1]$ and the set of rational numbers $\Q$ with the Euclidean topology. For any space $X$, let $\co{X}$ denote the collection of clopen subsets of $X$. The cardinality of a set $A$ will be denoted by $|A|$. A set $A$ is countable if $|A|\leq\omega$.

Let $X$ be a $T_1$ space. The Vietoris topology on $\CL{X}$ is the topology generated by the sets of the form
\begin{eqnarray}
U\sp+&=&\{A\in CL(X):A\subset U\},\nonumber\\
U\sp-&=&\{A\in CL(X):A\cap U\neq\emptyset\},\nonumber
\end{eqnarray}
where $U$ is an open subset of $X$. It is easy to see that a basis of the Vietoris topology consists of the collection of sets of the form
$$
\la U_0,\ldots U_n \ra=\{A\in CL(X):A\subset U_0\cup\cdots\cup U_n\textrm{ and if }i\leq n, U_i\cap A\neq\emptyset\},
$$
where $n<\omega$ and $U_0,\ldots, U_n$ are nonempty open subsets of $X$. For $n\in\N$, the hyperspace $\F[n]{X}$  is called the $n$-th symmetric product of $X$. Notice that $\F[1]{X}$ is homeomorphic to $X$ under the map $\{x\}\mapsto x$. We will use the following straightforward generalization of \cite[13.3]{ill-nad} several times.

\begin{lemma}\label{associated}
Let $X$ and $Y$ be Hausdorff spaces and $f:X\to Y$ be a continuous function. Define $f\sp\ast:\K{X}\to\K{Y}$ by $f\sp\ast(T)=f[T]$. Then $f\sp\ast$ is a continuous function.
\end{lemma}

Let $X$ be a Tychonoff space. A subset $A$ of a space $X$ is $C\sp\ast$-embedded if every bounded real-valued continuous function defined on $A$ can be extended to $X$. A zero set of $X$ is a set of the form $f\sp{\leftarrow}(0)$, where $f$ is a continuous real-valued function defined on $X$; a cozero set of $X$ is the complement of a zero set of $X$. In a Tychonoff space, cozero sets form a basis for its topology. 

\begin{lemma}\label{cozerohyperspace}
If $U$ is a cozero set of a Hausdorff space $X$, then $U\sp+\cap \K{X}$ and $U\sp-\cap \K{X}$ are cozero sets of $\K{X}$.
\end{lemma}
\begin{proof}
Let $f:X\to I$ be such that $U=f^{\leftarrow}[(0,1]]$. Consider the continuous function $f\sp\ast:\K{X}\to\K{I}$ from Lemma \ref{associated}. The functions $\min{}:\K{I}\to I$ and $\max{}:\K{I}\to I$ that take each closed subset of $I$ to its minimal and maximal elements, respectively, are easily seen to be continuous. Finally, notice that
\begin{eqnarray}
U^+\cap\K{X}&=&(\min{}\circ f\sp\ast)^{\leftarrow}[(0,1]],\nonumber\\
U^-\cap\K{X}&=&(\max{}\circ f\sp\ast)^{\leftarrow}[(0,1]],\nonumber
\end{eqnarray}
which completes the proof.
\end{proof}

Recall that a Tychonoff space is pseudocompact if every locally finite collection of open sets is finite. Let $X\subset I\sp\kappa$ for some cardinal $\kappa\geq 1$ and let $\pi_A:I\sp\kappa\to I\sp A$ be the projection for each $\emptyset\neq A\subset\kappa$. We say that $X$ is $\omega$-dense in $I\sp\kappa$ if every time $N$ is a countable nonempty subset of $\kappa$ it follows that $\pi_N[X]=I\sp N$.

\begin{lemma}\cite[Proposition 1]{shakh}\label{omegadense}
Let $X$ be a dense subspace of $I^\kappa$ for some cardinal $\kappa\geq 1$. Then $X$ is pseudocompact if and only if $X$ is $\omega$-dense in $I^\kappa$.
\end{lemma}

\begin{lemma}\cite[Theorem 1.3]{keesling}\label{pseudosym}
Let $X$ be a Tychonoff space and $n\in\N$. Then $\F[n]{X}$ is pseudocompact if and only if $X\sp n$ is pseudocompact.
\end{lemma}

For each space $X$ and $p\in X$, recall the quasicomponent of $X$ at $p$ is the closed subset
$$
\q{X}{p}=\bigcap\{U\in\co{X}:p\in U\}.
$$
We can define by transfinite recursion the $\alpha$-quasicomponent of $X$ at $p$, $\q[\alpha]{X}{p}$, in the following way.

$$
\begin{array}{lcll}
\q[0]{X}{p} & = & X,&\\
\q[\alpha+1]{X}{p} & = & \q{\q[\alpha]{X}{p}}{p},& \textrm{for each ordinal } \alpha,\\
\q[\beta]{X}{p} & = & \bigcap_{\alpha<\beta}{\q[\alpha]{X}{p}},& \textrm{for each limit ordinal } \beta.
\end{array}
$$

We call $\nc{X}{p}=\min\{\alpha:\q[\alpha+1]{X}{p}=\q[\alpha]{X}{p}\}$ the non-connectivity index of $X$ at $p$. If $X$ is hereditarily disconnected and $p\in X$, then $\nc{X}{p}=\min\{\alpha:\q[\alpha]{X}{p}=\{p\}\}$. Notice that if $X$ is hereditarily disconnected and $|X|>1$, then $X$ is totally disconnected if and only if $\nc{X}{p}=1$ for every $p\in X$.

If $X$ is any space (no separation axioms required), we can define a quotient space $\qs{X}$ by shrinking each quasicomponent of $X$ to a point (that is, $\qs{X}=\{\q{X}{p}:p\in X\}$ with the quotient topology). Observe that $\qs{X}$ is a Hausdorff totally disconnected space.

Recall a space $X$ is scattered if for every nonempty $Y\subset X$, the set of isolated points of $Y$ is nonempty.

\begin{lemma}\label{scattered}
If $X$ and $Y$ are compact Hausdorff spaces, $X$ is scattered and $Y$ is a continuous image of $X$, then $Y$ is also scattered.
\end{lemma}
\begin{proof}
Let $f:X\to Y$ be continuous and onto. Assume $K\subset Y$ is nonempty and does not have isolated points. By taking closure, we may assume that $K$ is closed. Using the Kuratowski-Zorn lemma, we can find a closed subset $T\subset X$ that is minimal with the property that $f[T]=K$. Since $X$ is scattered, there exists an isolated point $t\in T$ of $T$. Notice that $K-\{f(t)\}\subset f[T-\{t\}]$. Also, since $K$ has no isolated points, $K-\{f(t)\}$ is dense in $K$. But $T-\{t\}$ is compact so it follows that $K\subset f[T-\{t\}]$. This contradicts the minimality of $T$ so such $K$ cannot exist.
\end{proof}

\section{$P$-points in symmetric products}\label{ppoints}

In this section we show how to detect $P$-points in symmetric products.

\begin{lemma}\label{disjointbasis}
Let $X$ be a Hausdorff space, $n<\omega$, $A\in \F[n+1]{X}-\F[n]{X}$ (where $\F[0]{X}=\emptyset$) and $\U$ an open set in $\CL{X}$ such that $A\in\U$. Then there exist $U_0,\ldots,U_n$ pairwise disjoint nonempty open sets such that
$$
A\in\la U_0,\ldots,U_n\ra\subset\U.
$$
\end{lemma}
\begin{proof}
Let $A=\{x_0,\ldots,x_n\}$. Take $V_0,\ldots,V_n$ pairwise disjoint open subsets of $X$ such that $x_k\in V_k$ for each $k\leq n$. Consider now a basic open set
$$
A\in\la W_0,\ldots,W_s\ra\subset\U\cap\la V_0\ldots, V_n\ra.
$$
For each $k\leq n$ let $U_k=V_k\cap(\bigcap\{W_r:x_k\in W_r\})$. Notice that $U_0,\ldots,U_n$ are pairwise disjoint open sets such that
$$
A\in\la U_0,\ldots,U_n\ra\subset\la W_0,\ldots,W_s\ra\subset\U,
$$
which completes the proof.
\end{proof}

\begin{propo}\label{ppoint}
Let $X$ be a Hausdorff space and $A\in \F{X}$. The following conditions are equivalent
\begin{itemize}
\item[(a)] $A$ is a $P$-point of $\F{X}$,
\item[(b)] $A$ is a $P$-point of $\F[n]{X}$ for each $n\geq|A|$,
\item[(c)] every $x\in A$ is a $P$-point of $X$.
\end{itemize}
\end{propo}
\begin{proof}
Let $A=\{x_0,\ldots,x_m\}$. The implication $(a)\Rightarrow(b)$ is clear because the property of being a $P$-point is hereditary to subspaces. Assume $A$ is a $P$-point of $\F[m+1]{X}$. Let $\{U_i:i<\omega\}$ be a collection of open subsets of $X$ such that $x_0\in\bigcap_{i<\omega}{U_i}$. Take $W_0,\ldots,W_m$ pairwise disjoint open subsets of $X$ such that $x_i\in W_i$ for $j\leq m$. For each $i<\omega$, define
$$
\U_i=\la U_i\cap W_0,W_1,\ldots,W_m\ra.
$$
Since $A$ is a $P$-point in $\F[m+1]{X}$, by Lemma \ref{disjointbasis}, there is a collection $V_0,\ldots, V_m$ consisting of pairwise disjoint open subsets of $X$ such that
$$
A\in\la V_0,\ldots,V_m\ra\subset\bigcap_{i<\omega}{\U_i}.
$$ 

We may assume $x_j\in V_j$ for each $j\leq m$. We now prove $V_0\subset\bigcap_{i<\omega}{U_i}$. Take $y\in V_0$ and consider the element $B=\{y,x_1,\ldots,x_m\}\in\la V_0,\ldots, V_m\ra$. Since $B\in\U_i$ for each $i<\omega$, we get $y\in\bigcap_{i<\omega}{U_i}$. This proves that $x_0$ is a $P$-point of $X$ and by similar arguments, each point of $A$ is a $P$-point of $X$. This proves $(b)\Rightarrow(c)$.

Now, let $\{\U_i:i<\omega\}$ be a collection of open subsets of $\F{X}$ that cointain $A$ and assume each point of $A$ is a $P$-point of $X$. Using Lemma \ref{disjointbasis}, for each $i<\omega$ one may define a collection $U(0,i),\ldots, U(m,i)$ consisting of pairwise disjoint open subsets of $X$  such that for each $j\leq m$, $x_j\in U(j,i)$ and $\la U(0,i),\ldots,U(m,i)\ra\subset \U_i$. Each point of $A$ is a $P$-point so we may take, for each $j\leq m$, an open subset $U_j$ of $X$ such that $x\in U_j\subset\bigcap_{i<\omega}{U(j,i)}$. Thus, 
$$
A\in\la U_0,\ldots, U_m\ra\subset\bigcap_{i<\omega}{\U_i},
$$
which proves $(c)\Rightarrow(a)$.
\end{proof}

\begin{ex}
We construct a homogeneous $P$-space with no isolated points using hyperspaces. Let $X=\{\alpha+1:\alpha<\omega_1\}\cup\{0,\omega_1\}$ as a subspace of the LOTS $\omega_1+1$. So $X$ is a $P$-space in which all its points except for $\omega_1$ are isolated. Let 
$$
Y=\{A\in \F{X}:\omega_1\in A\},
$$
which is a $P$-space (Proposition \ref{ppoint}). Let $A\in Y$ and let $\la U_0,\ldots,U_k\ra$ be a basic open neighborhood of $A$. We may assume that $U_0,\ldots,U_k$ are pairwise disjoint (Lemma \ref{disjointbasis}) and $\omega_1\in U_0$. Let $\alpha\in U_0-\{\omega_1\}$, then $A\neq A\cup\{\alpha\}\in\la U_0,\ldots, U_n\ra$. Thus, $Y$ has no isolated points.

To prove the homogeneity of $Y$ it is sufficient to prove the following:
\begin{itemize}
\item[(1)] if $A,B\in Y$ are such that $|A|=|B|$, then there exists a homeomorphism $H:Y\to Y$ such that $H(A)=B$,
\item[(2)] for every $n\in\N$, there are $A,B\in Y$ such that $|A|+1=|B|=n+1$ and a homeomorphism $H:Y\to Y$ such that $H(A)=B$.
\end{itemize}

For (1), let $h:X\to X$ be a bijection such that $h[A]=B$ and $h(\omega_1)=\omega_1$. Define $H(P)=h[P]$ for every $P\in Y$. 

For (2), let $H:Y\to Y$ be defined by
$$
H(P)=\left\{\begin{array}{cc}
P-\{0\},&\textrm{ if } 0\in P,\\
P\cup\{0\},&\textrm{ if } 0\notin P.
\end{array}\right.
$$

Then $H$ is a homeomorphism such that for each $A\in Y$ with $0\notin P$, $|H(A)|=|A|+1$. It follows that $Y$ is homogeneous.

\raggedleft{\qedsymbol}
\end{ex}

\section{Hyperspaces that are $F$-spaces}\label{fspaces}

An \emph{$\mathit{F}$-space} is a Tychonoff space in which every cozero set is $C\sp\ast$-embedded. See problem 6L in \cite{porterwoods} for properties of $F$-spaces. All Tychonoff $P$-spaces are $F$-spaces but there are connected $F$-spaces (for example, $\beta{[0,1)}-[0,1)$ by \cite[6L(2)]{porterwoods} and \cite[6AA(2)]{porterwoods}). We may also consider $F\sp\prime$-spaces, that is, Tychonoff spaces in which each pair of disjoint cozero sets have disjoint closures (see Definition 8.12 in \cite{gil-hen}). Notice every $F$-space is an $F\sp\prime$-space. The main result of this section, Theorem \ref{hypfspace}, says that a hyperspace can only be an $F\sp\prime$-space when it is a $P$-space and thus disconnected.

\begin{fact}\label{fact1}
If $X$ is an infinite Hausdorff space, then $\CL{X}$ contains a convergent sequence.
\end{fact}
\begin{proof}
Let $N=\{x_n:n<\omega\}$ be an countable infinite subspace of $X$. If $A_m=\{x_n:n\leq m\}$ for $m<\omega$, then $\{A_m:m<\omega\}$ is a sequence that converges to $\cl{N}$.
\end{proof}

\begin{fact}\label{fact2}
If $X$ is an $F\sp\prime$-space, then $X$ does not contain convergent sequences
\end{fact}
\begin{proof}
If $N=\{x_n:n<\omega\}$ is a faithfully indexed sequence that converges to $x_0$, let $U, V$ be disjoint cozero sets of $X$ such that $\{x_{2n}:n\in\N\}\subset U$ and $\{x_{2n-1}:n\in\N\}\subset V$. Then $x_0\in\cl{U}\cap\cl{V}$.
\end{proof}

Thus, $\CL{X}$ is not an $F\sp\prime$-space unless $X$ is finite. So it is left to know when $\K{X}$, $\F{X}$ and the symmetric products can be $F\sp\prime$-spaces. 

Along with $F$-spaces and $F\sp\prime$-spaces, we may consider other classes of spaces between discrete spaces and $F$-spaces. A space is \emph{extremally disconnected} if every open set has open closure. A \emph{basically disconnected} space is a space in which every cozero set has open closure. So we may consider discrete spaces, $P$-spaces, extremally disconnected spaces (ED), basically disconnected spaces (BD), $F$-spaces and $F\sp\prime$-spaces. The diagram below (taken from \cite{gil-hen}) shows by an arrow which of these properties implies another. 

$$
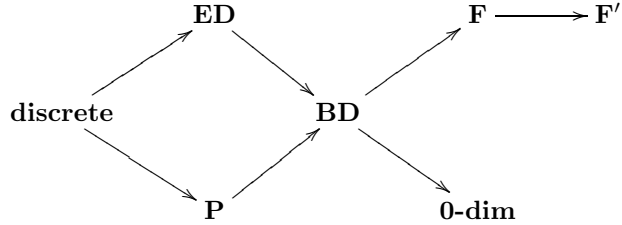
\xymatrix{
& \mathbf{ED} \ar[dr] & & \mathbf{F}\ar[r]& \mathbf{F\sp\prime}\\
\textbf{discrete} \ar[ur]\ar[dr] & & \mathbf{BD} \ar[ur]\ar[dr] & & \\
& \mathbf{P} \ar[ur] & & \mathbf{0}\textbf{-dim} &
}
$$

Theorem \ref{hypfspace} implies that $P$-spaces, BD spaces, $F$-spaces and $F\sp\prime$-spaces coincide for the hyperspaces $\K{X}$, $\F{X}$ and the symmetric products in the realm of Tychonoff spaces. Before heading on to prove this, let us show that the hyperspaces we are considering are extremally disconnected if and only if they are discrete, even in the realm of Hausdorff spaces.

\begin{propo}
Let $X$ be a Hausdorff space and $\F[2]{X}\subset\h\subset \K{X}$. Then $\h$ is extremally disconnected if and only if $X$ is discrete.
\end{propo}
\begin{proof}
Clearly, $X$ discrete implies $\h$ discrete. So assume that $X$ is not discrete, take a non-isolated point $p\in X$ and consider $\mathcal{Z}$ the set of all collections $\mathcal{G}$ such that the elements of $\mathcal{G}$ are pairwise disjoint nonempty open subsets of $X$ and if $U\in\mathcal{G}$, then $p\notin\cl{U}$. By the Kuratowski-Zorn Lemma, we can consider a $\subset$-maximal element $\mathcal{M}\in\mathcal{Z}$. Since $X$ is Hausdorff, $\bigcup\mathcal{M}$ is dense in $X$.

Let $\U=\bigcup\{U\sp+\cap\h:U\in\mathcal{M}\}$. Let $\mathcal{N}$ be the filter of open neighborhoods of $p$. For each $W\in\mathcal{N}$, there must be $U_W,V_W\in\mathcal{M}$ such that $U_W\neq V_W$, $W\cap U_W\neq\emptyset$ and $W\cap V_W\neq\emptyset$. Let $\V=\bigcup\{\la U_W,V_W\ra\cap\h:W\in\mathcal{N}\}$. Then, $\U$ and $\V$ are pairwise disjoint nonempty open subsets of $\h$ but $\{p\}\in\cl[\h]{\U}\cap\cl[\h]{\V}$.
\end{proof}

We now return to $F\sp\prime$-spaces. It was shown by Michael (Theorem 4.9 of \cite{michael}) that $\K{X}$ is Tychonoff if and only if $X$ is Tychonoff so we may assume that $X$ is Tychonoff for the rest of this section. Alternatively, this follows from Lemma \ref{cozerohyperspace}.

\begin{remark}\label{fspacefinite}
Let $X$ be an infinite Tychonoff space. If $\K{X}$ is an $F$-space, then $\K{X}=\F{X}$.
\end{remark}
\begin{proof}
If $Y\in \K{X}-\F{X}$, then $\CL{Y}\subset\K{X}$ contains a convergent sequence by Fact \ref{fact1}. This contradicts Fact \ref{fact2}.
\end{proof}

\begin{lemma}\cite[1W(2)]{porterwoods}\label{zeroisclopen} A Tychonoff space is a $P$-space if and only if every zero set is clopen.\end{lemma} 

\begin{propo}\label{fisp}
Let $X$ be a Tychonoff space and let $F_2(X)\subset\h\subset \K{X}$. If $\h$ is an $F\sp\prime$-space, then $X$ is a $P$-space.
\end{propo}
\begin{proof}
Let us assume $X$ is not a $P$-space, by Lemma \ref{zeroisclopen}, we may assume there is a continuous function $f:X\to I$ such that $Z=f\sp\leftarrow(0)$ is not clopen. Let $p\in Z-\inte{Z}$ and consider the following two statements:
\begin{itemize}
\item[(E)] There is a neighborhood $U$ of $p$ with $f[U]\subset\{0\}\cup\{\frac{1}{2m}:m\in\N\}$.
\item[(O)] There is a neighborhood $V$ of $p$ with $f[V]\subset\{0\}\cup\{\frac{1}{2m-1}:m\in\N\}$.
\end{itemize}

Notice that since $p\notin\inte{Z}$, we cannot have (E) and (O) simultaneaously. Assume without loss of generality that (E) does not hold. For each $m\in\N$, let $U_m=f\sp\leftarrow[(\frac{1}{2m+2},\frac{1}{2m})]$. Then $\{U_m:m\in\N\}$ is a collection of pairwise disjoint cozero sets. Observe that every neighborhood of $p$ intersects some $U_m$. Also, $f\sp\leftarrow[[0,\frac{1}{2m+2})]$ is a neighborhood of $p$ that misses $U_m$. Thus,
$$
(\ast)\ p\in\cl{\bigcup{\{U_m:m\in\N\}}}-\bigcup\{\cl{U_m}:m\in\N\}.
$$

Consider the sets:
\begin{eqnarray}
\U&=&\bigcup\{U_m\sp+\cap\h:m\in\N\},\nonumber\\
\V&=&\bigcup\{\la U_m,U_k\ra\cap\h:m,k\in\N,m\neq k\},\nonumber
\end{eqnarray}
these are nonempty pairwise disjoint cozero sets by Lemma \ref{cozerohyperspace}. By $(\ast)$, it follows that $\{p\}\in\cl[\h]{\U}\cap\cl[\h]{\V}$, so $\h$ is not an $F\sp\prime$-space.
\end{proof}

This allows us to give the next result.

\begin{thm}\label{hypfspace}
Let $X$ be a Tychonoff space and $\F[2]{X}\subset\h\subset \K{X}$. Then the following are equivalent:
\begin{itemize}
\item[(a)] $X$ is a $P$-space,
\item[(b)] $\h$ is a $P$-space,
\item[(c)] $\h$ is an $F\sp\prime$-space.
\end{itemize}
\end{thm}
\begin{proof}
First, assume (a). By Lemma \ref{ppoint}, $\F{X}$ is a $P$-space and by Remark \ref{fspacefinite}, $\K{X}=\F{X}$ so $\h\subset\K{X}$ is a $P$-space. So (b) holds. That (b) implies (c) is well-known and (c) implies (a) by Proposition \ref{fisp}.
\end{proof}

\section{Some spaces such that $\K{X}=\F{X}$}\label{compactisfinite}

We can generalize the techniques for $F\sp\prime$-spaces to another class of spaces $X$ such that $\K{X}=\F{X}$. We consider the case of weak $P$-spaces.

We call a point $p$ in a space $X$ a weak $P$-point of $X$ if for every countable set $N\subset X-\{p\}$ we have $p\notin\cl{N}$. A weak $P$-space is a space $X$ in which all its points are weak $P$-points of $X$.

\begin{fact}\label{fact3}
If $X$ is a weak $P$-space, then $\K{X}=\F{X}$.
\end{fact}
\begin{proof}
If $X$ is a weak $P$-space, then every countable subset of $X$ is closed and discrete. If $K\subset X$ is compact and infinite, it contains a countable infinite discrete subset $N\subset K$ and if $x\in\cl{N}-N$, then $N\cup\{x\}$ is countable but not discrete.
\end{proof}

\begin{remark}\label{weakpconnected}
The property of being a weak $P$-space does not imply disconnectedness. Shakhmatov gave in \cite{shakh} an example of a connected, pseudocompact, weak $P$-space in which all its countable subsets are $C\sp\ast$-embedded.
\end{remark}

We have the following results for weak $P$-spaces, analogous to those of $P$-spaces.

\begin{propo}\label{weakP1}
Let $X$ be a Hausdorff space and $A\in \F{X}$. The following conditions are equivalent
\begin{itemize}
\item[(a)] $A$ is a weak $P$-point in $\F{X}$,
\item[(b)] $A$ is a weak $P$-point in $\F[n]{X}$ for each $n\geq|A|$,
\item[(c)] every $x\in A$ is a weak $P$-point of $X$.
\end{itemize}
\end{propo}
\begin{proof}
Let $A=\{x_0,\ldots,x_m\}$. Notice $(a)\Rightarrow(b)$ because being a weak $P$-point is hereditary to subspaces.

To prove $(b)\Rightarrow(c)$, assume $x_0$ is not a weak $P$-point of $X$. Let $D=\{y_k:k<\omega\}\subset X-\{x_0\}$ be such that $x_0\in\cl{D}$. Define $B_k=\{y_k,x_1,\ldots,x_m\}\in\F[m+1]{X}$ for each $k<\omega$. Then, $\{B_k:k<\omega\}\subset\F[m+1]{X}-\{A\}$ and $A\in\cl[{\F[m+1]{X}}]{\{B_k:k<\omega\}}$.

Now we prove $(c)\Rightarrow(a)$. Assume $(c)$ and take $\{B_k:k<\omega\}\subset\F{X}-\{A\}$. For each $k<\omega$, choose $t(k)\in\{0,\ldots,m\}$ such that $x_{t(k)}\notin B_k$. Define $E_r=\{k<\omega:t(k)=r\}$ for each $r\leq m$. So given $r\leq n$, $x_r\notin\bigcup\{B_k:k\in E_r\}$. Since $\bigcup\{B_k:k\in E_r\}$ is countable, there exists an open subset $U_r$ with $x_r\in U_r$ and $U_r\cap(\bigcup\{B_k:k\in E_r\})=\emptyset$. Finally, let $\U=\la U_0,\ldots, U_n\ra$. Then $A\in\U$ and $\U\cap\{B_k:k<\omega\}=\emptyset$. 
\end{proof}

\begin{thm}\label{weakP2}
Let $X$ be a Hausdorff space. Then the following are equivalent
\begin{itemize}
\item[(a)] $X$ is a weak $P$-space,
\item[(b)] $\K{X}$ is a weak $P$-space,
\item[(c)] $\F{X}$ is a weak $P$-space,
\item[(d)] $\F[n]{X}$ is a weak $P$-space for some $n\in\N$.
\end{itemize}
\end{thm}
\begin{proof}
If we assume $(a)$, by Fact \ref{fact3} we have $\K{X}=\F{X}$, which is a weak $P$-space by Proposition \ref{weakP1}. Clearly, $(b)$ implies $(c)$ and  $(c)$ implies $(d)$. Finally, $(d)$ and $(a)$ are equivalent by Proposition \ref{weakP1}.
\end{proof}

\begin{ex}
Call a Tychonoff space a \emph{Shakhmatov space} if it is a pseudocompact, connected weak $P$-space. In \cite{shakh}, Shakhmatov gave an example of a subspace $\s\subset I\sp\cont$, where $\cont=2\sp\omega$, that is a Shakhmatov space. Moreover, $\s$ is $\omega$-dense in $I\sp\cont$.

Let $\h$ be either a symmetric product of $\s$ or one of the hyperspaces $\K{\s}$ or $\F{\s}$. By Theorem \ref{weakP2}, $\h$ is also a weak $P$-space and by Theorem \ref{michael1}, it is also connected. It is interesting to ask whether $\h$ is a Shakhmatov space. Notice that $\K{X}=\F{X}$ for every Shakhmatov space $X$ (Fact \ref{fact3}).

Since $\s$ is $\omega$-dense in $I^\cont$, it is easy to see that for any cardinal $\kappa$, $\s^\kappa$ is $\omega$-dense in $(I\sp\cont)\sp\kappa=I\sp{\cont\times\kappa}$. By Lemmas \ref{omegadense} and \ref{pseudosym}, $\F[n]{\s}$ is a Shakhmatov space for every $n\in\N$. However, $\F{X}$ is pseudocompact if and only if $X$ is finite by Lemma \ref{finitenotpseudo} below. Thus, $\K{X}$ is never a Shakhmatov space.

\raggedleft{\qedsymbol}
\end{ex}

\begin{lemma}\label{finitenotpseudo}
If $X$ is an infinite Tychonoff space, $\F{X}$ is not pseudocompact.
\end{lemma}
\begin{proof}
Let $\{U_n:n<\omega\}$ be a family of pairwise disjoint nonempty open subsets of $X$. If $\U_n=\la U_0,\ldots, U_n\ra$, then $\{\U_n\cap\F{X}:n<\omega\}$ is an infinite locally finite family of open nonempty subsets of $\F{X}$.
\end{proof}

For sake of completeness, we show that condition $\K{X}=\F{X}$ behaves well under the operation of taking hyperspace in the following way.

\begin{propo}
If $X$ is a Hausdorff space, then $\K{X}=\F{X}$ if and only if every compact subset of $\K{X}$ is finite (that is, $\K{\K{X}}=\F{\K{X}}$).
\end{propo}
\begin{proof}
First, assume $\K{X}=\F{X}$, and let $\C\subset\K{X}$ be compact. Write $\C=\bigcup_{n\in\N}\C_n$ where $\C_n=\C\cap\F[n]{X}$. Notice each $\C_n$ is compact because $\F[n]{X}$ is closed in $\F{X}$.

\begin{claim}[$\!\!$]
Each $\C_n$ is finite.
\end{claim}

Fix $n\in\N$. To prove the claim, consider the natural identification $\pi:{}\sp{n}X\to\F[n]{X}$ that sends each $n$-tuple to the set of its coordinates. Also, consider $\pi_k:{}\sp{n}X\to X$ the projection onto the $k$th-coordinate. Since $\pi$ is perfect, the set $K_n=\pi_k[\pi\sp\leftarrow[\C_n]]$ is a compact subset of $X$ and thus, finite. Now, $\pi\sp\leftarrow[\C_n]\subset K_1\times\cdots\times K_n$ so $\C_n$ must also be finite. This proves the Claim.

By the Claim, $\C$ is a compact Hausdorff countable space. Since the weight of an infinite compact Hausdorff space is less or equal to its cardinality (\cite[3.1.21]{engelking}), $\C$ is a compact metric space. Assume $\C$ is infinite, then we can find a faithfully indexed sequence $\{A_n:n<\omega\}\subset\C$ such that $A_0=\lim A_n$. 

Let $A_0=\{x_0,\ldots,x_s\}$ and take $U_0,\ldots, U_s$ pairwise disjoint open sets such that $x_i\in U_i$ for $i\leq s$. We may thus assume that for every $n<\omega$, $A_n\in\la U_0,\ldots, U_s\ra$. For each $n\in\N$, let $k_n\leq s$ be such that $A_n\cap U_{k_n}\neq\{x_{k_n}\}$, we may assume without loss of generality that $k_n=0$ for every $n\in\N$. Let
$$
Y=\bigcup\{A_n\cap U_0:n<\omega\}.
$$

First, if $Y$ is finite, there is an open set $V$ such that $V\cap Y=\{x_0\}$, so the neighborhood $\la V\cap U_0,U_1,\ldots,U_s\ra$ intersects the sequence only in $A_0$, which contradicts the convergence of the $A_n$. Thus, $Y$ is infinite. We now prove that $Y$ converges to $x_0$. Let $V$ be an open set such that $x_0\in V$. Let $k<\omega$ be such that $A_n\in\la V\cap U_0,U_1,\ldots, U_s\ra$ for each $n\geq k$. From this it follows that the set
$$
Y-\bigcup\{A_n\cap U_0:n<k\}
$$
is a cofinite subset of $Y$ contained in $V$. Thus, $Y$ is a nontrivial convergent sequence in $X$. This contradiction implies $\C$ is finite.

The other implication follows from the fact that $X$ is homeomorphic to $\F[1]{X}\subset\K{X}$.
\end{proof}

We end the discussion by showing that weak $P$-spaces are not the only ones in which the equality compact=finite holds.

\begin{ex}\label{ex1}
Let $X=\omega\cup P$, where $P$ is the set of weak $P$-points of $\omega\sp\ast$. It is a famous result of Kunen (\cite{kunen}) that $P$ is a dense subset of $\omega\sp\ast$ of cardinality $2^{2^\omega}$. We claim that $\K{X}=\F{X}$. Every infinite compact space contains a separable compact subspace, so it is sufficient to show that the closure of every infinite countable subset $N\subset X$ is not compact. Since $P$ is a weak $P$-space closed in $X$, $\cl{N\cap P}=\cl[P]{N\cap P}=N\cap P$ that is compact if and only if it is finite. Thus, we may assume $N\subset\omega$. Since $\omega\sp\ast-P$ is also dense in $\omega\sp\ast$, $\cl[\beta\omega]{N}-X\neq\emptyset$. It easily follows that $\cl{N}$ is not compact. Notice that $X$ is not a weak $P$-space because $N$ is dense in $X$.

\raggedleft{\qedsymbol}
\end{ex}

Observe that the space $X$ from Example \ref{ex1} is extremally disconnected because it is a dense subspace of $\beta\omega$. We now present an example of a space whose compact subspaces are finite but it is not an $F\sp\prime$-space. Recall a space has countable cellularity if every collection of pairwise disjoint nonempty open subsets is countable. For the proof of the following fact follow the hint in \cite[6L(8)]{porterwoods}.

\begin{fact}\label{cell}
Every $F\sp\prime$-space of countable cellularity is extremally disconnected.
\end{fact}

\begin{ex}
Let $\omega=\bigcup\{A_n:n<\omega\}$ be a partition in infinite subsets. Let $\filt_0$ be the Frechet filter  (or any filter that contains it) and
$$
\filt=\{B\subset\omega:\{n<\omega:A_n-B\textrm{ is finite}\}\in\filt_0\}.
$$

Define the space $X=\omega\cup\{\filt\}$ where every point of $\omega$ is isolated and the neighborhoods of $\filt$ are of the form $\{\filt\}\cup A$ with $A\in\filt$. 

Any infinite compact subspace of $X$ must be a convergent sequence. Let $S\subset\omega$ be infinite. If there exists $m<\omega$ such that $S\cap A_m$ is infinite, let $R=\omega-A_m$. If for each $n<\omega$, $|S\cap A_n|<\omega$ holds, let $R=\omega-S$. In both cases $R\in\filt$ and $S-R$ is infinite, so $S$ cannot converge to $\filt$.

Also, notice that $X$ is an $F\sp\prime$-space if and only if it is extremally disconnected ($X$ has countable cellularity, use Fact \ref{cell}) and it is easy to see this happens if and only if $\filt$ is an ultrafilter. To see $\filt$ is not an ultrafilter, for each $n<\omega$, let $A_n=P_n\cup Q_n$ be a partition in infinite subsets. Then $P=\bigcup\{P_n:n<\omega\}\notin\filt$ and $Q=\bigcup\{Q_n:n<\omega\}\notin\filt$ but $\omega=P\cup Q$. Thus, $\filt$ is not prime so it is not an ultrafilter.

Thus, $X$ is a space in which all compact subsets are finite but it is not an $F\sp\prime$-space.

\raggedleft{\qedsymbol}
\end{ex}

\section{Hereditary disconnectedness}\label{herdisc}

Our first result gives a method to locate connected sets in a hyperspace.

\begin{lemma}\label{01}
Let $X$ be a Hausdorff space. Assume there is a $K\in\K{X}$ such that for every $U\in\co{X}$ with $K\subset U$ we have $X=U$. Then 
$$
\C=\{K\cup\{x\}:x\in X\}
$$
is a connected subset of $\K{X}$.
\end{lemma}
\begin{proof}
Let $\U$ and $\V$ be open subsets of $\K{X}$ such that $K\in\U$, $\C\subset\U\cup\V$ and $\C\cap\U\cap\V=\emptyset$. Let $U=\{x\in X:K\cup\{x\}\in\U\}$ and $V=X-U$. Clearly, $K\subset U$, we now prove that $U$ is clopen.

First, we prove every point $x\in U$ is in the interior of $U$, we have two cases. If $x\in K$, let $n<\omega$ and $U_0,\ldots,U_n$ be open subsets of $X$ such that
$$
K\in\la U_0\ldots,U_n\ra\subset\U.
$$
Notice that $x\in K\subset U_0\cup\ldots\cup U_n\subset U$. If $x\notin K$, let $V_0,\ldots,V_m,W$ be open subsets of $X$ such that $K\subset V_0\cup\ldots\cup V_m$, $x\in W$, $W\cap (V_0\cup\ldots\cup V_m)=\emptyset$ and $K\cup\{x\}\in\la V_0,\ldots,V_m,W\ra\subset\U$. Then, $x\in W\subset U$.

Now let $x\in V$, then $K\cup\{x\}\in \C-\U\subset\V$. Let $V_0,\ldots,V_m,W$ be open subsets of $X$ such that $K\subset V_0\cup\ldots\cup V_m$, $x\in W$, $W\cap (V_0\cup\ldots\cup V_m)=\emptyset$ and $K\cup\{x\}\in\la V_0,\ldots,V_m,W\ra\subset\V$. Then $x\in W\subset V$. This proves $V$ is open and thus, $U$ is closed. 

Therefore, $U$ is clopen and contains $K$ so by hypothesis $U=X$. But this implies that $\C\subset\U$. Then $\C$ is connected.
\end{proof}

Using Lemma \ref{01}, we give a modification of Example 1.1 of \cite{pol} showing there was no need to add a Cantor set to the original space.

\begin{ex}\label{02}
Let $C\subset I$ be the Cantor set constructed by removing middle-thirds of intervals in the usual way, let $Q\subset C$ be the set of endpoints of the Cantor set and $P=C-Q$. For each $c\in C$, let

$$
L_c=
\begin{cases}
\{c\}\times ([0,1)\cap\Q),&\textrm{ if }c\in Q,\\
\{c\}\times ([0,1)-\Q),&\textrm{ if }c\in P.
\end{cases}
$$

Let $\f=\bigcup\{L_c:c\in C\}$. Notice $\f$ is homeomorphic to the Knaster-Kuratowski fan with its top point removed (see \cite[1.4.C]{engelkingdim}). It is easy to see that $\f$ is hereditarily disconnected. Let $\pi:\f\to C$ be the projection to the first coordinate (in the plane). We now prove:

\begin{claim}[1]
There is a compact $G\subset \f$ such that if $c\in C$, $|\pi^{\leftarrow}(c)\cap G|=1$.
\end{claim}

To prove Claim 1, let $D=Q\cup[\Q\cap(I-C)]$ which is a countable dense subset of $I$. It is a well-known fact that there is a homeomorphism $h:I\to [0,\frac{1}{2}]$ such that $f[D]=\Q\cap[0,\frac{1}{2}]$. Let $G=f\restriction_{C}\subset C\times[0,\frac{1}{2}]$, the graph of the function $f$ restricted to the Cantor set. Claim 1 follows.

\begin{claim}[2]
Let $A,B$ closed sets of the plane such that $A\cap B\cap \f=\emptyset$, $G\subset A$ and $\f\subset A\cup B$. Then $\f\cap B=\emptyset$.
\end{claim}

To prove Claim 2, let $\Q\cap [0,1)=\{q_n:n<\omega\}$ be an enumeration. For each $n<\omega$, let $P_n=C\times\{q_n\}$ and $K_n=\pi[A\cap B\cap P_n]$. Notice that $K_n$ is a compact subset of $P$ because $A\cap B\cap \f=\emptyset$ and $\f\cap P_n=Q\times\{q_n\}$. 

Moreover, $K_n$ is nowhere dense in $P$. To see this, assume $W$ is a nonempty regular open subset of $C$ with $W\cap P\subset K_n$. We have $\cl[C]{W\cap P}=\cl[C]{W}$ because $P$ is dense in $C$. Let $x\in W\cap Q$, then $x\in W\subset\cl[C]{W\cap P}\subset K_n$. So $(x,q_n)$ is a point of $F$ whose first coordinate is in $K_n$, this implies $(x,q_n)\in A\cap B\cap \f$, a contradiction.

Since $P$ is completely metrizable, it is a Baire space and the set $Z=P-(\bigcup_{n<\omega}K_n)$ is a dense open subset of $P$. Fix $c\in Z$. Then for each $n<\omega$, $(c,q_n)\notin A\cap B$. Since $L_c$ is dense in $\{c\}\times I$, $\{c\}\times I\subset A\cup B$. Now, $\{c\}\times I$ is connected so either $\{c\}\times I\subset A$ or $\{c\}\times I\subset B$. Since $(c,f(c))\in G\subset A$, we necessarily have $\{c\}\times I\subset A$. But this implies that $\bigcup\{L_c:c\in Z\}$ is a dense subset of $\f$ contained in $A$. Then $\f\subset A$ so $\f\cap B=\emptyset$. This proves Claim 2.

By Claim 2 and Lemma \ref{01}, $\C=\{G\cup\{x\}:x\in \f\}$ is a connected subset of $\K{\f}$ with more than one point. We have proved that $\K{\f}$ is not hereditarily disconnected.

\raggedleft{\qedsymbol}

\end{ex}

A question one may ask is if $\K{X}$ is hereditarily disconnected when $X$ is an hereditarily disconnected space that is the union of two totally disconnected subspaces. Consider space $\f$ from Example \ref{02}: we can write $\f$ as the union of two totally disconnected subspaces $\f=[\f\cap\Q^2]\cup[\f-\Q^2]$ and $\K{\f}$ is not hereditarily disconnected. So we need more conditions that ensure that $\K{X}$ is hereditarily disconnected. Our Main Theorem shows that under certain conditions $\K{X}$ is hereditarily disconnected. Before proving it, we isolate two technical Lemmas we will use often.

\begin{lemma}\label{03}
Let $X$ be a $T_1$ space, $T\subset X$ such that
\begin{itemize}
\item[(a)]for every $x\in X-T$ there is a $W\in\co{X}$ such that $x\in W$ and $W\cap T=\emptyset$,
\item[(b)]$X-T$ is totally disconnected.
\end{itemize}
Let $\C\subset\K{X}$ be connected. Then the following holds
\begin{quote}
$(\ast)$ if $Y_1,Y_2\in\C$, then $Y_1-T=Y_2-T$.
\end{quote}
\end{lemma}
\begin{proof}
For the sake of producing a contradiction, let us assume $(\ast)$ does not hold for some $Y_1,Y_2\in\C$. Let, without loss of generality, $y\in Y_2-T$ be such that $y\notin Y_1$. For each $x\in Y_1-T$, let $U_x\in\co{X}$ such that $x,y\in U_x$ and $U_x\cap T=\emptyset$, this can be done by $(a)$. Since $U_x\subset X-T$ is totally disconnected by $(b)$, let $V_x\in\co{U_x}$ be such that $x\notin V_x$ and $y\in V_x$. Let $W_x=X-V_x$, observe both $V_x,W_x\in\co{X}$.

Notice that $T\cup\{x\}\subset W_x$ and $y\notin W_x$. By compactness, there is a finite set $\{x_0,\ldots,x_n\}\subset Y_1-T$ such that $Y_1\cup T\subset W_{x_0}\cup\cdots\cup W_{x_n}$. So $W=W_{x_0}\cup\ldots\cup W_{x_n}$ is a clopen subset of $X$ such that $Y_1\in W^+$ and $Y_2\notin W^+$. But this contradicts the connectedness of $\C$ so $(\ast)$ holds. 
\end{proof}

\begin{lemma}\label{04}
Let $X$ be a $T_1$ space, $T\subset X$ a closed subset and $\emptyset\neq\C\subset\K{X}$ such that
\begin{itemize}
\item[(a)] if $Y_1,Y_2\in\C$, then $Y_1-T=Y_2-T$,
\item[(b)] if $Y\in\C$, then $Y\cap T\neq\emptyset$.
\end{itemize}
Define $\Phi:\C\to\K{T}$ by $\Phi(Y)=Y\cap T$. Then $\Phi$ is a well-defined, injective and continuous function.
\end{lemma}
\begin{proof}
The function $\Phi$ is well-defined by $(b)$ and is injective by $(a)$, we only have to prove the continuity. Let $Y_0\in\C$. Define $Z=Y_0-T$. Notice that by $(a)$, $Z=Y-T$ for every $Y\in \C$. If $Z=\emptyset$, $\Phi$ is an inclusion that is clearly continuous so assume $Z\neq\emptyset$.

Let $\U$ be an open subset of $\K{T}$ with $\Phi(Y_0)\in\U$. We now prove there is an open subset $\V$ of $\K{X}$ such that $Y_0\in\V$ and $\Phi[\V\cap\C]\subset\U$. We may assume that $\U=\la U_1,\ldots,U_n\ra$ where $U_1,\ldots,U_n$ are nonempty open subsets of $T$.

Let $V_0=X-T$. For $1\leq m\leq n$, let $V_m$ be an open subset of $X$ such that $V_m\cap T=U_m$ and if $U_m\cap\cl{Z}=\emptyset$, then also $V_m\cap\cl{Z}=\emptyset$. Let $\V=\la V_0,V_1,\ldots,V_n\ra$, clearly $Y_0\in\V$.

Let $Y\in\V\cap\C$. First, if $y\in\Phi(Y)$, then $y\in V_m\cap T$ for some $1\leq m\leq n$. Thus, $\Phi(Y)\subset U_1\cup\ldots U_n$. Now, let $1\leq m\leq n$. If there is a point $y\in U_m\cap \cl{Z}\neq\emptyset$, then since $\cl{Z}\subset Y$, $y\in U_m\cap\Phi(Y)$. If $U_m\cap\cl{Z}=\emptyset$, let $y\in Y\cap V_m$ so that $y\in U_m\cap\Phi(Y)$. In both cases, $U_m\cap\Phi(Y)\neq\emptyset$. This shows $\Phi(Y)\in\U$ and completes the proof.

\end{proof}

The Main Theorem will be proved in two steps. The first step is to add just one point to a totally disconnected space.

\begin{propo}\label{05}
Let $X$ be a Hausdorff hereditarily disconnected space and $p\in X$ be such that $X-\{p\}$ is totally disconnected. Then $\K{X}$ is hereditarily disconnected.
\end{propo}
\begin{proof}
Start with a connected subset $\C\subset\K{X}$. By considering iterated quasicomponents, we shall prove that $|\C|=1$.

For each ordinal $\alpha$, let $T_\alpha=\q[\alpha]{X}{p}$ and $\Gamma=\nc{X}{p}$. Notice that $\{T_\alpha:\alpha<\Gamma\}$ is a strictly decreasing family of closed subsets of $X$ that contain $p$ and $T_\Gamma=\{p\}$. We prove the following two properties by transfinite induction on $\alpha$:

\begin{quote}
$(\ast)_\alpha$ If $Y_1,Y_2\in\C$, then $Y_1-T_\alpha=Y_2-T_\alpha$.
\end{quote} 

\vskip0.2cm

\begin{quote}
$(\star)_\alpha$ If there exists $Y_0\in\C$ such that $Y_0\cap T_\alpha=\emptyset$, then $\C=\{Y_0\}$.
\end{quote}

To prove $(\ast)_0$, just apply Lemma \ref{03} to the pair of spaces $T_0\subset X$. Now, let $Y_0$ as in $(\star)_0$, so one can find $W\in\co{X}$ such that $Y_0\subset W$ and $T_0 \cap W=\emptyset$. But then $W\sp+$ is a clopen set so $Y\in W\sp+$ for all $Y\in\C$. By $(\ast)_0$, we get $(\star)_0$.

Now assume $(\ast)_\gamma$ and $(\star)_\gamma$ for every $\gamma\leq\beta$. We now prove $(\ast)_{\beta+1}$ and $(\star)_{\beta+1}$.

We first consider $(\ast)_{\beta+1}$. If there exists $Y_0\in\C$ such that $Y_0\cap T_\beta=\emptyset$, by $(\star)_\beta$, we have $\C=\{Y_0\}$ and $(\ast)_{\beta+1}$ is clearly true. So assume that every $Y\in\C$ intersects $T_\beta$. By Lemma \ref{04}, the function $\Phi_\beta:\C\to\K{T_\beta}$ defined by $\Phi_\beta(Y)=Y\cap T_\beta$ is continuous and injective. Let $\C_\beta=\Phi_\beta[\C]$. Using Lemma \ref{03} for the pair of spaces $T_{\beta+1}\subset T_\beta$ and the connected subset $\C_\beta$ we get for every $Y_1,Y_2\in\C$, $(Y_1\cap T_\beta)-T_{\beta+1}=(Y_2\cap T_\beta)-T_{\beta+1}$. By $(\ast)_\beta$, this implies $(\ast)_{\beta+1}$.

Notice that if there is a $Y_0\in\C$ such that $Y_0\cap T_\beta=\emptyset$, then $(\star)_{\beta}$ implies $(\star)_{\beta+1}$ so assume for every $Y\in\C$, $Y\cap T_\beta\neq\emptyset$. Again we may consider $\Phi_\beta$ and $\C_\beta$ as in the former paragraph. Let $Y_0\in\C$ such that $Y_0\cap T_{\beta+1}=\emptyset$. Then one can find $W\in\co{T_\beta}$ such that $\Phi_\beta[Y_0]\subset W$ and $W\cap T_{\beta+1}=\emptyset$. So $W\sp+$ is a clopen set that intersects the connected set $\C_\beta$, therefore, $\Phi_\beta[Y]\in W\sp+$ for every $Y\in\C$. By $(\ast)_{\beta+1}$ we conclude $(\star)_{\beta+1}$.

We have left to prove $(\ast)_\beta$ and $(\star)_\beta$ for $\beta$ a limit ordinal but these proofs follow from $(\ast)_\gamma$ and $(\star)_\gamma$ for each $\gamma<\beta$ using that $T_\beta=\bigcap_{\gamma<\beta} T_\gamma$.

Observe that $(\ast)_\Gamma$ means that if $Y_1,Y_2\in\C$, then $Y_1-\{p\}=Y_2-\{p\}$. By $(\star)_\Gamma$ it easily follows that $|\C|=1$. So $\K{X}$ is hereditarily disconnected.

\end{proof}

We now procede to prove the main result.

\begin{theproof}
Let $\C\subset\K{X}$ be a connected subset. Denote by $\pi:X\to X/F$ the quotient map and denote by $\widetilde{F}$ the unique point in $\pi[F]$. Let $\mathcal{D}=\{\pi[C]:C\in\C\}$, this set is connected because $\mathcal{D}=\pi\sp\ast[\C]$ where $\pi\sp\ast:\K{X}\to\K{X/F}$ is the continuous function defined in Lemma \ref{associated}. Using Proposition \ref{05} for $\widetilde{F}\in X/F$ it follows that $\mathcal{D}=\{T\}$ for some $T\in\K{X/F}$. If $\widetilde{F}\notin T$, since $\pi$ is inyective in $X-F$, $|\C|=1$. If $\widetilde{F}\in T$, then $Y\cap F\neq\emptyset$ for every $Y\in\C$. Thus, by Lemma \ref{04}, the function $\Phi:\C\to\K{F}$ given by $\Phi(Y)=Y\cap F$ is continuous and injective. But $F$ is totally disconnected, so by Theorem \ref{michael2}, $\K{F}$ is totally disconnected. Thus, $|\C|=|\mathcal{D}|=1$. 

\raggedleft{\qedsymbol}
\end{theproof}

A natural question here is if the converse to the Main Theorem is true. That is, assume $X=Y\cup F$ where both $Y,F$ are totally disconnected, $F$ is closed and $\K{X}$ is hereditarily disconnected, is it true that the quotient $X/F$ must also be hereditarily disconnected? When $F$ is compact, the answer is in the affirmative (Corollary \ref{06}) but it may not be in general (Case 2 of the Example from Section \ref{example}).

\begin{coro}\label{06}
Let $X$ be a Hausdorff space. Assume $X=Y\cup T$ where both $Y$ and $T$ are totally disconnected and $T$ is compact. Then $\K{X}$ is hereditarily disconnected if and only if the quotient space $X/T$ is hereditarily disconnected.
\end{coro}
\begin{proof}
Let $\pi:X\to X/T$ be the quotient and $\widetilde{T}$ the unique point in $\pi[T]$. If $X/T$ is hereditarily disconnected, then $\K{X}$ is hereditarily disconnected by the Main Theorem. If $X/T$ is not hereditarily disconnected, let $R\subset X/T$ be a connected subset with more than one point. Clearly $\widetilde{T}\in R$. Let $F=\pi\sp\leftarrow[R]$, notice $T\subset F$. Define $\C=\{T\cup\{x\}:x\in F\}$ which is connected by Lemma \ref{01}. Moreover, $|\C|>1$ because $R\neq\{\widetilde{T}\}$.
\end{proof}

Let us prove that if $\K{X}$ has a connected subset with more than one point, then it must also contain a cannonical one in some sense.

\begin{propo}\label{07}
Let $X$ be a Hausdorff hereditarily disconnected space. If $\C\subset\K{X}$ is a connected set with more than one point and $K\in\C$, then there is a closed subset $F\subset X$ with $K\subsetneq F$ such that the set $\D=\{K\cup\{x\}:x\in F\}$ is connected and $|\D|>1$.
\end{propo}
\begin{proof}
Consider the set
$$
\mathcal{Z}=\{Z\subset X:Z\textrm{ is closed and for every }Y\in\C,Y\subset Z\}.
$$
By the Kuratowski-Zorn lemma, there exists a $\subset$-minimal element $F\in\mathcal{Z}$. Notice $K\subset F$. Let $\D=\{K\cup\{x\}:x\in F\}$.

Assume $K=F$. Since $K$ is compact, it is zero-dimensional and $\K{K}$ is also zero-dimensional (Theorem \ref{michael2}). Then $\C$ is a connected subset of $\K{K}$, this implies $|\C|=1$. This is a contradiction so we have $K\subsetneq F$, which implies $|\D|>1$.

Let $q:F\to\qs{F}$ be the quotient map onto the space of quasicomponents of $F$. Consider the continuous function $q\sp\ast:\K{F}\to\K{\qs{F}}$ from Lemma \ref{associated}. Since $\K{\qs{F}}$ is totally disconnected (Theorem \ref{michael2}), $q\sp\ast[\C]=\{T\}$ for some compact $T\subset\qs{F}$. Then $G=q\sp\leftarrow[T]$ is such that $G\subset F$ and $\C\subset\K{G}$. By minimality of $F$, $F=G$. Thus, $q[K]=q\sp\ast(K)=T=q[F]=\qs{F}$ so $K$ intersects every quasicomponent of $F$. From this and Lema \ref{01} it easily follows that $\D$ is a connected subset of $\K{X}$.

\end{proof}

To finish with this section, we generalize the ``countable'' in Theorem 1.3 of \cite{pol} to ``scattered''. We start with a useful remark that will help with the proof.

\begin{remark}\label{converseconnected}
If $F$ is hereditarily disconnected and $K\subset F$ is a compact subset such that $\{K\cup\{x\}:x\in F\}$ is connected, then $K$ intersects every quasicomponent of $F$.
\end{remark}

\begin{thm}\label{08}
Let $X$ be a Hausdorff hereditarily disconnected space. If $\C\subset\K{X}$ is connected and there exists $T\in\C$ that is scattered, then $|\C|=1$.
\end{thm} 
\begin{proof}
Assume that $\C\subset\K{X}$ is connected and $|\C|>1$. By Proposition \ref{07}, we may assume $\C=\{T\cup\{x\}:x\in F\}$ for some closed subset $F\subset X$ such that $T\subset F$. 

We now define a descending transfinite sequence of closed sets $F_\alpha$ ($\alpha$ an ordinal) in the following way. We first take $F_0=F$. Assume we have already defined $F_\alpha$. Let $q_\alpha:F_\alpha\to\qs{F_\alpha}$ be the quotient map and let $U_\alpha\subset\qs{F_\alpha}$ be the set of isolated points of $\qs{F_\alpha}$. Define $F_{\alpha+1}=F_\alpha-q_\alpha\sp\leftarrow[U_\alpha]$. Finally, if $\beta$ is a limit ordinal, let $F_\beta=\bigcap_{\alpha<\beta}F_\alpha$.

We also define for each ordinal $\alpha$, $T_\alpha=F_\alpha\cap T$ (so that $T_0=T$) and
$$
\C_\alpha=\{T_\alpha\cup\{x\}:x\in F_\alpha\}.
$$

By transfinite induction on $\alpha$ we shall prove the following properties
\begin{itemize}
\item[$(0)_\alpha$] If for each $\beta<\alpha$ we have $F_\beta\neq\emptyset$, then for each $\beta<\alpha$, $F_\alpha\subsetneq F_\beta$.
\item[$(1)_\alpha$] $(a)$ For every $Y_1,Y_2\in\C$, $Y_1-F_\alpha=Y_2-F_\alpha$,\\ $(b)$ For each $Y\in\C$, $T_\alpha\subset Y$,\\$(c)$ If $F_\alpha\neq\emptyset$, the function $\Phi_\alpha:\C\to\K{F_\alpha}$ given by $\Phi_\alpha(Y)=Y\cap F_\alpha$ is well-defined, continuous and injective. Moreover, $\C_\alpha=\Phi_\alpha[\C]$.
\item[$(2)_\alpha$] $q_\alpha[T_\alpha]=\qs{F_\alpha}$.
\end{itemize}

First, notice that $(1)_\alpha$ implies $(2)_\alpha$. To see this, observe that $(1c)_\alpha$ implies $\C_\alpha$ is connected. By Remark \ref{converseconnected}, we get $(2)_\alpha$.

Clearly, $(0)_0$ and $(1)_0$ are true. Assume $(0)_\alpha$, $(1)_\alpha$ and $(2)_\alpha$ hold. 

Since $T_\alpha$ is a compact Hausdorff scattered space, it must be $0$-dimensional so by $(2)_\alpha$ and Lemma \ref{scattered}, $\qs{F_\alpha}$ is a compact $0$-dimensional scattered space. Thus, if $F_\alpha\neq\emptyset$, then also $U_\alpha\neq\emptyset$ and since $q_\alpha$ is onto, $F_{\alpha+1}\subsetneq F_\alpha$. From this $(0)_{\alpha+1}$ follows.

Observe that for each $x\in U_\alpha$, $q_\alpha\sp\leftarrow(x)$ is a clopen quasicomponent of $F_\alpha$, so it must be an isolated point $\{y\}$. By $(2)_\alpha$, $y\in T_\alpha$. We have obtained
$$
(\star)_\alpha\ q_\alpha\sp\leftarrow[U_\alpha]\subset T_\alpha.
$$
So we can write 
$$
(\ast)_\alpha\ \C_\alpha=\{T_\alpha\cup\{x\}:x\in F_{\alpha+1}\}\cup\{T_\alpha\}.
$$

We now prove $(1)_{\alpha+1}$.

First, let $Y_1,Y_2\in\C$ and $x\in Y_1-F_{\alpha+1}$. If $x\notin F_\alpha$, by $(1a)_\alpha$, $x\in Y_2-F_\alpha\subset Y_2-F_{\alpha+1}$. If $x\in F_\alpha$, by $(\ast)_\alpha$, we get $T_\alpha\cup\{x\}=T_\alpha\cup\{y\}$ for some $y\in F_{\alpha+1}$ or $T_\alpha\cup\{x\}=T_\alpha$. Notice $x\neq y$ so it must be that $x\in T_\alpha$. Thus, $x\in T\subset Y_2$. We have obtained that $Y_1-F_{\alpha+1}\subset Y_2-F_{\alpha+1}$ and by a similar argument, $Y_2-F_{\alpha+1}\subset Y_1-F_{\alpha+1}$. This proves $(1a)_{\alpha+1}$.

Condition $(1b)_{\alpha+1}$ is true because of $(1b)_\alpha$ and the fact that $T_{\alpha+1}\subset T_\alpha$.

Assume $F_{\alpha+1}\neq\emptyset$. Notice that by $(2)_\alpha$, $T_{\alpha+1}=T_\alpha\cap F_{\alpha+1}\neq\emptyset$. Then, $(1b)_{\alpha+1}$ implies that for each $Y\in\C$, $Y\cap F_{\alpha+1}\neq\emptyset$. Using this, $(1a)_{\alpha+1}$ and Lemma \ref{04} it can be shown that $\Phi_{\alpha+1}$ is a well-defined, continuous and injective function. By similar arguments and $(1c)_\alpha$, we may define a function $\Psi:\C_\alpha\to\K{F_{\alpha+1}}$ by $\Psi(Y)=Y\cap F_{\alpha+1}$ which is continuous and injective. Moreover, the following diagram commutes:

$$
\xymatrix{
\C\ar[d]_{\Phi_\alpha}\ar[rd]\sp{\Phi_{\alpha+1}}&\\
\C_\alpha\ar[r]_-{\Psi}&\K{F_{\alpha+1}}
}
$$

From equation $(\ast)_\alpha$, we deduce $\Phi_{\alpha+1}[\C]=\Psi[\C_\alpha]=\C_{\alpha+1}$. This proves $(1c)_{\alpha+1}$.

Now, let us assume $(0)_\alpha$, $(1)_\alpha$ and $(2)_\alpha$ for all $\alpha<\gamma$ for some limit ordinal $\gamma$. Assume $F_\alpha\neq\emptyset$ for each $\alpha<\gamma$. Fix $\alpha<\gamma$. From $F_\gamma\subset F_{\alpha+1}\subset F_\alpha$ we see that $F_\gamma\neq F_\alpha$. Otherwise, $F_{\alpha+1}=F_\alpha$, which contradicts $(0)_{\alpha+1}$. Thus, we get $(0)_{\gamma}$.

From $F_\gamma=\bigcap_{\alpha<\gamma}F_\alpha$, $T_\gamma=\bigcap_{\alpha<\gamma}{T_\alpha}$ and $(1a)_\alpha,(1b)_\alpha$, one can easily deduce $(1a)_\gamma$ and $(1b)_\gamma$. Assume $F_\gamma\neq\emptyset$.  By $(2)_\alpha$, $T_\alpha\neq\emptyset$ for each $\alpha<\gamma$. Then by $(0)_\alpha$, the $T_\alpha$, with $\alpha<\gamma$, form a strictly descending chain of compact nonempty sets, this implies $T_\gamma=\bigcap_{\alpha<\gamma}{T_\alpha}\neq\emptyset$. By $(1a)_\gamma$ and $(1b)_\gamma$, we can apply Lemma \ref{04} to conclude that $\Phi_\gamma$ is well-defined, continuous and injective. Then, it is easy to see that $\Phi_\gamma[\C]=\C_\gamma$. This proves $(1c)_\gamma$.

This completes the induction. Notice that by $(0)_\alpha$, one can define
$$
\Gamma=\min\{\alpha:F_\alpha=\emptyset\}.
$$
 
One can show, using $(2)_\alpha$ and the compactness of the $T_\alpha$, that $\Gamma=\Lambda+1$ for some ordinal $\Lambda$. Observe that $F_\Gamma=F_\Lambda-q_\Lambda\sp\leftarrow[U_\Lambda]$, so every point of $F_\Lambda$ is isolated. Then, $T_\Gamma$ is a discrete compact set and thus finite. By $(2)_\Lambda$, $\qs{F_\Lambda}$ must be finite and since it is a space of quasicomponents, $F_\Lambda=\qs{F_\Lambda}$. Thus, $\C_\Lambda=\{T_\Lambda\}$. But $\C_\Lambda$ is the injective image of $\C$ under $\Phi_\Lambda$. This contradicts $|C|>1$. Therefore, $|\C|=1$.

\end{proof}

It is inmediate that the following holds

\begin{coro}
Let $X$ be a Hausdorff space. Then the following are equivalent
\begin{itemize}
\item[(a)] $X$ is hereditarily disconnected,
\item[(b)] for some (equivalently, for each) $n\in\N$, $\F[n]{X}$ is hereditarily disconnected,
\item[(c)] $\F{X}$ is hereditarily disconnected.
\end{itemize}
\end{coro}

\section{An example for the Main Theorem}\label{example}

In this section, we present two examples related to the Main Theorem. Notice that the statement of Corollary \ref{06} contains a converse of the statement of the Main Theorem for the case that $T$ is a compact space. The first example (Case 1 below) is an example of this inverse implication. The second example (Case 2 below) shows that one cannot obtain an inverse of the statement of the Main Theorem relaxing the requirement of compactness of $T$ to that of being a closed subset of $X$.

Consider $\cs$, the Cantor set as a product (where $2=\{0,1\}$ is a discrete space), and $[0,\infty]$ with the interval topology (that is, the Euclidean topology extended with a supremum $\infty$). Let $\pi:\cs\times[0,\infty]\to\cs$ and $h:\cs\times[0,\infty]\to[0,\infty]$ be the first and second projections, respectively. For each $i<\omega$, let $\rho_i:\cs\to 2$ be the projection onto the $i$-th coordinate. We will say a subset $A\subset\cs\times[0,\infty]$ is bounded if $\sup h[A]<\infty$ and unbounded if it is not bounded (thus, $h$ denotes the ``height''). Let $\phi:\cs\to[0,\infty]$ be the function
$$
\phi(t)=\sum_{m<\omega}\frac{t_m}{m+1}.
$$
We will consider the spaces $X=\{x\in\cs:\phi(x)<\infty\}$ and $X_0=\{(x,\phi(x)):x\in X\}$. In \cite[p. 600]{dijkstra}, Dijkstra shows that $X_0$ is homeomorphic to complete Erd\"os space. Moreover, this space has the following property

\begin{quote}
$(\nabla)$ every nonempty clopen subset of $X$ is unbounded.
\end{quote}

We will use the basis of $\cs$ formed by the clopen subsets of the form
$$
[a_0,\ldots, a_n]=\{x\in\cs:\rho_m(x)=a_m\textrm{ for all }m\leq n\}
$$
where $\{a_0,\ldots,a_n\}\subset\{0,1\}$.

Observe that both $X$ and $\cs-X$ are dense: for every open set of the form $[a_0,\ldots,a_n]$ we may choose $x,y\in [a_0,\ldots,a_n]$ such that $x_m=0=1-y_m$ for each $m>n$; then $x\in X$ and $y\in\cs-X$.

For each $K\subset\cs$ we define $K_0=K\times\{\infty\}$ and $Y=X_0\cup K_0$. Notice that since $\pi\restriction_{Y}$ is $\leq$2-to-1 and $\pi[Y]\subset\cs$ is $0$-dimensional, then $Y$ is hereditarily disconnected. By a similar argument, $X_0$ and $K_0$ are totally disconnected. We now analyze whether $\K{Y}$ is hereditarily disconnected for two specific examples for $K$.

\vskip0.2cm	
\noindent\underline{Case 1}. $K=\cs$.

\vskip0.2cm
First, $Y/K_0$ is connected: if $U\in\co{Y}$ is such that $K_0\subset U$, using the compactness of $K$ we get that $X-U$ is bounded, so $Y=U$ by $(\nabla)$. By Corollary \ref{06}, $\K{Y}$ is not hereditarily disconnected. Notice that in this case, $Y/K_0$ is homeomorphic to the space of Example 1.4.8 of \cite{engelkingdim}. By \cite[p. 600]{dijkstra}, $Y/K_0$ is also homeomorphic to the set of non-ordinary points of the Lelek fan.

\vskip0.2cm
\noindent\underline{Case 2}. $K=X$

\vskip0.2cm
First, we see that $Y/K_0$ is connected. Observe that since $K_0$ is not compact, $Y/K_0$ is not the same quotient as in Case 1 (it is not even first countable at the image of $K_0$). If $U\in\co{Y}$ is such that $K_0\subset U$ and $(x,\infty)\in K_0$, there exists $W\in\co{\cs}$ and $t\in[0,\infty)$ such that $(x,\infty)\in W\times(t,\infty]\subset U$. Thus, $V=(W\times [0,\infty])-U$ is a bounded clopen subset. By $(\nabla)$, $V=\emptyset$. Since $(x,\infty)$ was arbitrary, we get $U=Y$. Thus, $Y/K_0$ is connected. However, we cannot use Corollary \ref{06} because $K_0$ is not compact. In fact, we will show that $\K{Y}$ is hereditarily disconnected. Observe that the proof that $Y/K_0$ is connected can be modified to show that $Y$ is not totally disconnected, so it is not obvious that $\K{Y}$ is hereditarily disconnected.

A first attempt to prove that $\K{Y}$ is hereditarily disconnected could be showing that any compact subset of $Y$ is scattered and use Theorem \ref{08}. However, this is false. Recall that $\sum_{m=0}\sp\infty{\frac{1}{m+1}}=\frac{\pi\sp2}{6}<\infty$, then the subset
$$
P=\{(x,t)\in Y:\textrm{for each square-free }n\in\N, \rho_{n-1}(x)=0\}
$$
is bounded. As it is pointed out in \cite[p. 600]{dijkstra}, $P$ is homeomorphic to the Cantor set. Fortunately, every compact subset of $Y$ will be ``almost everywhere bounded'' in the sense of $(0)\sp\prime_\alpha$ below. We will follow the technique of Theorem \ref{08} to prove that $\K{Y}$ is hereditarily disconnected.

Assume that $\K{Y}$ contains a connected subset $\C$ with more than one point. We may assume by Proposition \ref{07} that $\C=\{T\cup\{x\}:x\in F\}$ for some closed $F\subset Y$ and some compact $T\subsetneq F$. We now construct a decreasing sequence of closed subsets $F_\alpha\subset Y$ for each ordinal $\alpha$. Start with $F_0=F$. If $F_\alpha$ has already been defined, let

$$
\begin{array}{cl}
U_\alpha= &\{x\in F_\alpha:\textrm{there is an open subset }U\subset\cs\textrm{ and }r\in(0,\infty)\textrm{ with }\\
&\ \pi(x)\in U\textrm{ such that if }y\in F_\alpha\cap\pi\sp\leftarrow[U],\textrm{ then }h(y)<r\textrm{ or }h(y)=\infty\}.
\end{array}
$$

Notice that $U_\alpha$ is open in $F_\alpha$. Moreover, if $x\in F_\alpha$ and $U$ is like in the definition above for $x$, then $F_\alpha\cap\pi\sp\leftarrow[U]\subset U_\alpha$. Thus

\begin{quote}
$(\star)_\alpha$ Let $x,y\in F_\alpha$ be such that $\pi(x)=\pi(y)$. Then $x\in U_\alpha$ if and only if $y\in U_\alpha$.
\end{quote}

So let $F_{\alpha+1}=F_\alpha-U_\alpha$, which is closed. Finally, if $\beta$ is a limit ordinal, let $F_\beta=\bigcap_{\alpha<\beta}F_\alpha$. We also define for each ordinal $\alpha$, $T_\alpha=T\cap F_\alpha$ and $\C_\alpha=\{T_\alpha\cup\{x\}:x\in F_\alpha\}$.

We now prove the following properties by transfinite induction.

\begin{itemize}
\item[$(1)_\alpha$] $(a)$ For every $Y\in\C$, $T_\alpha\subset Y$,\\ $(b)$ If $Y_1,Y_2\in\C$, $Y_1-F_\alpha=Y_2-F_\alpha$.\\ $(c)$ If $F_\alpha\ne\emptyset$, the function $\Phi_\alpha:\C\to\K{F_\alpha}$ given by $\Phi(Y)=Y\cap F_\alpha$ is well-defined, continuous and injective. Moreover, $\C_\alpha=\Phi_\alpha[\C]$. 
\item[$(2)_\alpha$] For each $x\in X$, $F_\alpha\cap\pi\sp\leftarrow(x)\neq\emptyset$ implies $T_\alpha\cap\pi\sp\leftarrow(x)\neq\emptyset$.
\item[$(3)_\alpha$] If $x\in\pi[U_\alpha]$, $F_\alpha\cap\pi\sp\leftarrow(x)=T_\alpha\cap\pi\sp\leftarrow(x)$.
\end{itemize}

We will procede in the following fashion:

\begin{itemize}
\item Step 1: $(1)_0$ is true.
\item Step 2: $(1{c})_\alpha$ implies $(2)_\alpha$ and $(3)_\alpha$ for each ordinal $\alpha$.
\item Step 3: $(1)_\alpha$ implies $(1)_{\alpha+1}$ for each ordinal $\alpha$.
\item Step 4: If $\beta$ is a limit ordinal, $(1)_\alpha$ for each $\alpha<\beta$ implies $(1)_\beta$.
\end{itemize}

This proof is very similar to that of Theorem \ref{08}, so we will omit some arguments when they follow in a similar way. Step 1 is clear, observe that $\Phi_0$ is the identity function. 

Proof of Step 2: Notice that if $F_\alpha=\emptyset$, $(2)_\alpha$ and $(3)_\alpha$ are true, so we may assume $F_\alpha\neq\emptyset$. Thus, $(1c)_\alpha$ implies $\C_\alpha$ is connected.

First, we prove $(2)_\alpha$. Let $x\in X$ and 
$$
(\bullet)\ Y\cap\pi\sp\leftarrow(x)=\{(x,t_0),(x,t_1)\}.
$$
Aiming towards a contradiction, assume $(x,t_0)\in F_\alpha$ and $T_\alpha\cap\pi\sp\leftarrow(x)=\emptyset$. Since $x$ is not in the compact set $\pi[T_\alpha]$, there is a $W\in\co{\cs}$ such that $x\in W$ and $W\cap\pi[T_\alpha]=\emptyset$. Since $T_\alpha\cup\{(x,t_0)\}\in(\pi\sp\leftarrow[W])\sp-$ and $T_\alpha\notin(\pi\sp\leftarrow[W])\sp-$, the clopen set $(\pi\sp\leftarrow[W])\sp-$ separates $\C_\alpha$. This contradiction shows that $(2)_\alpha$ holds.

Next, we prove $(3)_\alpha$. Let $x\in\pi[U_\alpha]$ be such that $\pi\sp\leftarrow(x)\cap F_\alpha\neq\emptyset$. Let us use equation $(\bullet)$ above. By $(2)_\alpha$ and the fact that $T_\alpha\subset F_\alpha$, we only have to show that the case when $\pi\sp\leftarrow(x)\cap F_\alpha=\pi\sp\leftarrow(x)$ and $\pi\sp\leftarrow(x)\cap T_\alpha=\{(x,t_0)\}$ is impossible. We will analyze when $t_0<\infty$, the other possibility being similar.

Since $x\in\pi[U_\alpha]$, there is an open subset $U\subset\cs$ and $r\in(0,\infty)$ such that if $y\in F_\alpha$ and $\pi(y)\in U$, then $h(y)\notin [r,\infty)$. Since $T_\alpha\cap(\cs\times[0,r])=R$ is a nonempty compact set and $x\notin\pi[R]$, there exists $W_0\in\co{\cs}$ such that $x\in W_0$ and $W_0\cap\pi[R]=\emptyset$. We may assume that $W_0\subset U$. Let
$$
W_1=(W_0\times [0,r])\cap F_\alpha=(W_0\times[0,r))\cap F_\alpha
$$
which is a clopen subset of $F_\alpha$. Further, $T_\alpha\cup\{(x,t_0)\}\in W_1\sp-$ and $T_\alpha\notin W_1\sp-$, this gives a separation of $\C_\alpha$. This is a contradiction so $(3)_\alpha$ follows.

Proof of Step 3: Assume $(1)_\alpha$. By Step 2, $(2)_\alpha$ and $(3)_\alpha$ hold. We may also assume that $F_{\alpha+1}\neq\emptyset$, otherwise $(1)_{\alpha+1}$ is clearly true. First we prove that

$$
(\ast)_\alpha\ \C_\alpha=\{T_\alpha\cup\{x\}:x\in F_{\alpha+1}\}\cup\{T_\alpha\}.
$$

The right side of $(\ast)_\alpha$ is clearly contained in the left side. Let $T_\alpha\cup\{x\}\in\C_\alpha$ with $x\in F_\alpha$. If $x\notin F_{\alpha+1}$, by $(3)_\alpha$, $x\in T_\alpha$. Thus, $T_\alpha\cup\{x\}=T_\alpha$ that is in the right side of $(\ast)_\alpha$. Thus, $(\ast)_\alpha$ follows.

We also need that $T_{\alpha+1}\neq\emptyset$. Let $x\in\pi[F_{\alpha+1}]$, by $(2)_\alpha$ there are $x_1,x_2\in F_\alpha\cap\pi\sp\leftarrow(x)$ such that $x_1\in F_{\alpha+1}$ and $x_2\in T_\alpha$. By $(\star)_\alpha$, $x_2\in T_{\alpha+1}$.

The remaining part of the argument is similar to that of Theorem \ref{08}, in the part where it is shown that $(1)_{\alpha+1}$ is a consequence of $(0)_\alpha$, $(1)_\alpha$ and $(2)_\alpha$. 

The proof of Step 4 is also similar to the part of Theorem \ref{08} where it is shown $(1)_\beta$ is the consequence of $(0)_\alpha$,$(1)_\alpha$ and $(2)_\alpha$ for all $\alpha<\beta$ when $\beta$ is a limit ordinal so we ommit it. This completes the induction.

The key to this example is the following statement:

\begin{quote}
$(0)_\alpha$ If $F_\alpha\neq\emptyset$, then $U_\alpha\neq\emptyset$.
\end{quote}

We shall use the technique Erd\"os used for for the proof of $(\nabla)$ (for the original Erd\"os space, see \cite{erdos}) to prove $(0)_\alpha$. Assume $F_\alpha\neq\emptyset$ but $U_\alpha=\emptyset$ for some $\alpha$. We now use induction to find elements $\{x_n:n<\omega\}\subset F_\alpha$, a strictly increasing sequence $\{s_n:n<\omega\}\subset\omega$, $y\in\cs-X$ and a decreasing sequence of open subsets $\{U_n:n<\omega\}$. For each $n<\omega$, call $t_n=\pi(x_n)$ and $y_n=\rho_n(y)$. We find all these with the following properties

\begin{itemize}
\item[$(i)$] $t_n\in U_n$,
\item[$(ii)$] for each $m\leq n$ and $r\leq s_n$, $\rho_r(t_m)=y_r$,
\item[$(iii)$] if $m<n$, then $m+h(x_m)<h(x_n)<\infty$,
\item[$(iv)$] if $m<n$, then $m+h(x_m)<\sum_{m=0}\sp{n+1}{\frac{y_m}{m+1}}<\infty$,
\item[$(v)$] $U_n=[y_0,\ldots,y_{s_n}]$.
\end{itemize}
 
For $n=0$ define $s_0=0$ and choose $x_0\in F_\alpha$ arbitrarily. Assume that we have the construction up to $n$. Since $x_n\notin U_\alpha$, there exists $x_{n+1}\in F_\alpha\cap\pi\sp\leftarrow[U_n]$ such that $n+h(x_n)<h(x_{n+1})<\infty$. Since
$$
\sum_{m<\omega}{\frac{\rho_m(t_{n+1})}{m+1}}=h(x_{n+1})<\infty,
$$
by the convergence of this series, there exists $s_{n+1}>s_n$ such that
$$
n+h(x_n)<\sum_{m=0}\sp{s_{n+1}}{\frac{\rho_m(t_{n+1})}{m+1}}.
$$

Define $y_m=\rho_m(t_{n+1})$ for $m\in\{s_n+1,\ldots,s_{n+1}\}$. Clearly, conditions $(i)-(v)$ hold. Notice that by $(iv)$, $\phi(y)=\infty$ so in fact $y\in\cs-X$.

By $(ii)$, $\{t_n:n<\omega\}$ converges to $y$. Moreover, by $(iii)$, $\{x_n:n<\omega\}$ converges to $(y,\infty)\notin Y$. Since $T_\alpha$ is compact, there exists $N<\omega$ such that for each $N\leq n<\omega$, $x_n\in F_\alpha-T_\alpha$.

Let $z_n=(t_n,\infty)$ for each $n<\omega$. If $N\leq n<\omega$ then by $(2)_\alpha$, $z_n\in T_\alpha$. But $\{z_n:N\leq n<\omega\}$ converges to $(y,\infty)\notin T_\alpha$, which is a contradiction. Thus, $(0)_\alpha$ follows.

Obseve that one may also use a similar argument to prove:

\begin{quote}
$(0)\sp\prime_\alpha$ $U_\alpha$ is dense in $F_\alpha$.
\end{quote}

We are now ready to produce a contradiction to the assumption that $\K{Y}$ is not hereditarily disconnected. By $(0)_\alpha$, we know that if $F_\alpha\neq\emptyset$, then $F_{\alpha+1}\subsetneq F_\alpha$. Thus, there exists 
$$
\Gamma=\min\{\alpha:F_\alpha=\emptyset\}.
$$
By $(2)_\alpha$ and a compactness argument, it can be proved that $\Gamma=\Lambda+1$ for some $\Lambda$. Then $U_\Lambda=F_\Lambda$, by $(3)_\Lambda$ this implies $F_\Lambda=T_\Lambda$. Thus $\C_\Lambda=\{T_\Lambda\}$. But $\Phi_\Lambda$ is an injective function by $(1c)_\Lambda$ so we have a contradiction. This contradiction proves that $\K{Y}$ is hereditarily disconnected.

\section{Final Remarks}

A substantial part of this paper was focused on giving conditions in $X$ so that $\K{X}$ is hereditarily disconnected. The Main Theorem and its Corollary \ref{06} are examples of this. We can also infer a little more by considering Theorem \ref{08}, Example \ref{02} and Case 2 from Section \ref{example}. It is clear that none of this results gives a complete solution to this problem. However, according to Proposition \ref{07} and Remark \ref{converseconnected}, we have the following characterization.

\begin{propo}\label{ugly}
Let $X$ be a hereditarily disconnected space. Then $\K{X}$ contains a connected set with more than one point if and only if there exists a closed subset $F\subset X$ and a compact subset $K\subsetneq F$ such that $K$ intersects every quasicomponent of $F$.
\end{propo}

However, this result is not something tangible in the following sense. Corollary \ref{06} says that if we want to know whether $\K{X}$ is hereditarily disconnected we just have to examine an specific space $X/T$. However, Proposition \ref{ugly} says we must look for some undetermined subsets $K$ and $F$.
 
\begin{ques}
Give tangible conditions on $X$ so that $\K{X}$ is hereditarily disconnected.
\end{ques}

\begin{ques}
Let $X$ be a homogeneous (topological group, perhaps) hereditarily disconnected space. Can one give some characterization of hereditarily disconnectedness of $\K{X}$ in terms of (iterated) quasicomponents and/or the space of quasicomponents $\qs{X}$?
\end{ques}

\end{document}